\documentclass[11pt]{amsart}
\usepackage{amsfonts,amssymb,graphicx, enumerate,youngtab}
\usepackage{dsfont}
\usepackage{amsmath}
\usepackage[usenames,dvipsnames]{pstricks}
\usepackage{epsfig}
\usepackage{multicol}
\usepackage{color, colortbl}

\usepackage{adjustbox}
\usepackage{cancel}
\usepackage{comment}

\numberwithin{equation}{section}

\newtheorem{theorem}{Theorem}[section]

\newtheorem{proposition}[theorem]{Proposition}

\theoremstyle{definition}
\newtheorem{definition}[theorem]{Definition}
\newtheorem{example}[theorem]{Example}

\theoremstyle{remark}

\newcommand{\CC}{\mathbb{C}}
\newcommand{\HH}{\mathbb{H}}

\newcommand{\ZZ}{\mathbb{Z}}

\numberwithin{equation}{section}

\begin{document}
\title{Singular polynomials for the rational Cherednik algebra for $G(r,1,2)$}

\author{Armin Gusenbauer}
\address{Universidad de Talca} 
\email{armingk@inst-mat.utalca.cl}

\footnote{I acknowledge the financial support of Conicyt Beca Doctorado Nacional folio 21120504. }

\begin{abstract}
We study the rational Cherednik algebra attached to the complex reflection group $G(r,1,2)$. Each irreducible representation
$S^{\lambda}$ of $G(r,1,2)$ corresponds to a standard module $\Delta(\lambda)$ for the rational Cherednik algebra. We give necessary and sufficient conditions for the existence of morphism between two of these modules and explicit formulas for them when they exist.
\end{abstract}

\maketitle
\section{Introduction}
 The rational Cherednik algebra $\mathbb{H}$ is an algebra attached to a complex reflection group $W$, depending on a set of parameters indexed by the conjugacy
classes of reflection in $W$. The algebra $\mathbb{H}$ possesses a triangular decomposition (\cite{D} and \cite{pEvG}) allowing the construction of induced modules called standard modules, and the Serre subcategory of $\HH$-mod generated by these, category $\mathcal{O}$, has been the object of intense study during the last fifteen years. Part of the structure of the category $\mathcal{O}$ is encoded by the homomorphisms between standard modules, and the classification and construction of these homomorphisms seems to be a difficult problem.

The first work on this problem is due to Dunkl \cite{Dun1}, \cite{Dun2}, who solved it for $W=S_n$ the symmetric group and codomain the standard module parabolically induced from the trivial representation. Subsequently Griffeth \cite{SG1} solved it for $W=G(r,1,n)$, but with a certain genericity condition in the parameters.  We will specialize to $W=G(r,1,2)$ and solve the problem without any restriction on the parameters.

The parameters space for $W=G(r,1,2)$ is $r$-dimensional with coordinates $c_0,d_0,d_1,...,d_{r-1}$ subject to the requirement
\begin{align*}
d_0+d_1+d_2+...+d_{r-1}=0.
\end{align*}
The irreducible representations of $G(r,1,n)$ are indexed by $r$-partitions of $n$. So for $n=2$ there are three kinds of irreducible representations $\left\{\lambda_i, \lambda^i , \lambda_{i,j} | 0\leq i\neq j \leq r-1\right\} $.
Our main theorem gives necessary and sufficient conditions for the existence of morphisms between the corresponding standard modules (see Theorem \ref{th1}).



For the necessary conditions we start by using Theorem 5.1 of \cite{GGAD}. For the sufficient conditions we construct the morphisms explicitly. This amounts to finding elements of the codomain that are annihilated by the Dunkl operators. In other words, we are looking for a generalized version of singular polynomials.

 For $G(r,1,2)$ the dimension of the homomorphism space between two standard modules is always at most two. The next theorem gives sufficient conditions for the dimension to be equal to 2 (we suspect that this the only way this can happen).
\begin{theorem}
If we have the conditions \begin{center}\begin{itemize}
                                          \item $d_i-d_k+c_0r=i-k+m_1r>0$
                                          \item $d_i-d_k-c_0r=i-k+m_2r>0$
                                          \item $d_j-d_i+c_0r=j-i+m_3r>0$
                                          \item $d_j-d_i-c_0r=j-i+m_4r>0$
                                        \end{itemize}
                                        \end{center}
where $m_i$ is a integer for $i=1,2,3,4$, then we have $$\mathrm{Dim}(\mathrm{Hom}(\Delta(\lambda_{i,k}),\Delta(\lambda_{i,j})))=2.$$
\end{theorem}

As we say before in order to prove our results we combine the necessary conditions from \cite{GGAD}  with explicit computations involving Dunkl operators acting on vector-valued polynomial functions. A standard teqnique for answering the questions we pose here is to aply the KZ-functor and use known results about Hecke algebras. The obstruction in our case is that we do not have good control over the KZ images of the standard modules (except for parameteres in a certain cone). 

One might hope that our results would compute the simple modules in category $\mathcal{O}$. However, it is quite rare that the radical of the standard module is generated by the singular polinomials it contains. For instance if the radical of every standard module is generated by singular polynomials then every simple object in category O has a BGG resolution by standard modules (Theorem 1.1 of \cite{GN} ).

 Category $\mathcal{O}$ is a highest weight category with $BGG$ reciprocity so by Lemma 4.5 of \cite{GGAD} it is equipped with a canonical coarsest order. In the example at the end of the paper we observe that this poset is graded and self-dual. This raises the question if this is always so and if there  is a structural reason for this phenomenon.
\section{Notation and preliminaries}
An $r$-$partition$ of $n$ is a sequence $\lambda=(\lambda^0,...,\lambda^{r-1})$ of partitions such that the sum of all the boxes of all the partitions is $n$.
A standard Young tableau $T$ on an $r$-partition $\lambda$ of $n$ is a filling of the boxes of the partitions $\lambda^0,...,\lambda^{r-1}$ with the integer $1,...,n$ in such a way that the entries within each partition $\lambda^i$ are increasing in the rows and the columns. For example for $n=2$ we have three kind of $r$-$partitions$ of 2. They are:\begin{enumerate}
   \item $\lambda_i=\left(\emptyset,...,\yng(2),...,\emptyset\right).$
   \item $\lambda^i=\left(\Yvcentermath1 \emptyset,...,\yng(1,1),...,\emptyset\right).$
   \item $\lambda_{i,j}=\left(\emptyset,...,\yng(1),...,\yng(1),...,\emptyset\right).$
 \end{enumerate}
 Where $0\leq i,j\leq {r-1}$ and the boxes are in position $i$ and $j$.
For partitions $\lambda_i$ and $\lambda^i$ there is one standard Young tableau associated and for $\lambda_{i,j}$ there are two:
\begin{multicols}{2}
\begin{itemize}
\item $\left(\emptyset,...,\young(1),...,\young(2),...,\emptyset\right)$
\item $\left(\emptyset,...,\young(2),...,\young(1),...,\emptyset\right)$
\end{itemize} 
\end{multicols}

Let be $W=G(r,1,2)$ the group of $2\times 2$ monomial matrices where each entry is a $r$-root of unity. The irreducible representations of $W$ may be parametrized by $r$-partitions of 2 in such a way that if $S^{\lambda}$ is the irreducible module corresponding to $\lambda$, then $S^{\lambda}$ has a basis $v_{T}$ indexed by $SYT(\lambda)$ (see Theorem 3.1 of \cite{SG1}).
\subsection{The rational Cherednik algebra for $G(r,1,2)$}
Let be $y_1=(1,0)$, $y_2=(0,1)$, $x_1=(1,0)^t$ and $x_2=(0,1)^t$, so that $y_1,y_2$ is the standard basis of  $\mathfrak{h}=\CC^2$ and $x_1,x_2$ is the dual basis of $\mathfrak{h}^\ast$.

Let $c_0,d_0,d_1,...,d_{r-1}\in \CC$. We define $d_i$ for all $i\in \ZZ$ with the equations:
\begin{align*}
\begin{array}{ccccc}
  d_0+d_1+...+d_{r-1}=0 & \mathrm{and} & d_i=d_j & \mathrm{if} & i=j \mathrm{mod} r
\end{array}
\end{align*}
The rational Cherednik algebra for $W=G(r,1,2)$ with parameters $c_0,d_1,...,d_{r-1}$ is the algebra generated by $\CC [x_1,x_2]$, $\CC [y_1,y_2]$ and $\overline{w}$ for $w\in W$ with the relations
$$\begin{array}{cccc}
   \bar{w}\bar{v}=\overline{wv} & \bar{w}x=(wx)\bar{w} & $ and $  & \bar{w}y=(wy)\bar{w}
 \end{array}$$
 for $w,v \in W$, $x \in \CC [x_1,x_2]$, and $y \in \CC [y_1,y_2]$,
\begin{align*}
y_ix_j=x_jy_i+c_0\displaystyle\sum_{l=0}^{r-1}\zeta^{-l}\overline{\zeta_i^ls_{ij}\zeta^{-j}_i}
\end{align*}
for $1\leq i \neq j\leq 2$, and
\begin{align*}
y_ix_i=x_iy_i+\kappa-\displaystyle \sum_{l=1}^{r-1}(d_j-d_{j-1})e_{ij}-c_0\displaystyle \sum_{j\neq i}\sum_{l=0}^{r-1}\overline{\zeta_i^ls_{ij}\zeta^{-j}_i}
\end{align*}
for $1\leq i\leq 2$, where $e_{ij}\in \CC W$ is the idempotent
\begin{align*}
e_{ij}=\displaystyle\frac{1}{r}\sum_{l=0}^{r-1}\zeta^{-lj}\overline{\zeta_i^l}.
\end{align*}

The PBW theorem (see for example \cite{SG4}) for $\HH$ assert that as $\CC$-vector spaces,
$$\HH\simeq \CC[x_1,x_2]\otimes_{\CC} \CC W \otimes_{\CC} \CC[y_1,y_2]$$
The following proposition is a particular case of Proposition 4.1 of \cite{SG1} when $n=2$.

 \begin{proposition} \label{propact}
The relations between $y_1$ and $y_2$ with an element of the form $x_1^nx_2^m$ are given by:

\begin{center}
$\begin{array}{ccl}
  y_1x_1^nx_2^m & = & x_1^nx_2^my_1+x_1^{n-1}x_2^m\left(n-\displaystyle\sum_{j=0}^{r-1}\frac{d_j}{r}\sum_{l=0}^{r-1}\zeta^{-lj}(1-\zeta^{-ln})\overline{\left(
                                                                                                                                               \begin{array}{cc}
                                                                                                                                                 \zeta^l & 0 \\
                                                                                                                                                 0 & 1 \\
                                                                                                                                               \end{array}
                                                                                                                                             \right)
 }\right) \\
  \hfill & \hfill & -c_0 \displaystyle\sum_{l=0}^{r-1}\frac{x_1^nx_2^m-\left(
                                                                                              \begin{array}{cc}
                                                                                                0 & \zeta^{l} \\
                                                                                                \zeta^{-l} & 0 \\
                                                                                              \end{array}
                                                                                            \right)\cdot x_1^nx_2^m
 }{x_1-\zeta^lx_2}\overline{\left(
                    \begin{array}{cc}
                      0 & \zeta^{l} \\
                      \zeta^{-l} & 0 \\
                    \end{array}
                  \right)}
\end{array}.$
\end{center}

\begin{center}
$\begin{array}{ccl}
         y_2x_1^nx_2^m & = & x_1^nx_2^my_2+x_1^{n}x_2^{m-1}\left(m-\displaystyle\sum_{j=0}^{r-1}\frac{d_j}{r}\sum_{l=0}^{r-1}\zeta^{-lj}(1-\zeta^{-lm})\overline{\left(
                                                                                                                                               \begin{array}{cc}
                                                                                                                                                 1 & 0 \\
                                                                                                                                                 0 & \zeta^l \\
                                                                                                                                               \end{array}
                                                                                                                                             \right)
 }\right) \\
         \hfill & \hfill & -c_0 \displaystyle\sum_{l=0}^{r-1}\frac{x_1^nx_2^m-\left(
                                                                                              \begin{array}{cc}
                                                                                                0 & \zeta^{-l} \\
                                                                                                \zeta^{l} & 0 \\
                                                                                              \end{array}
                                                                                            \right)\cdot x_1^nx_2^m
 }{x_2-\zeta^{l}x_1}\overline{\left(
                    \begin{array}{cc}
                      0 & \zeta^{-l} \\
                      \zeta^{l} & 0 \\
                    \end{array}
                  \right)}
       \end{array}.
$
\end{center}

\end{proposition}
\subsection{Standard modules for the rational Cherednik algerba}
Recall that the irreducible $\CC W$-modules $S^{\lambda}$ are parametrized by $r$-partition $\lambda$ of $2$. Define the standard module $\Delta (\lambda)$ to be the induced module
\begin{align*}
\Delta(\lambda)=\mathrm{Ind}^{\HH}_{\CC W\otimes \CC[y_1,y_2]}S^{\lambda}
\end{align*}
and we define the $\CC[y_1,y_2]$ action on $S^{\lambda}$ by
\begin{align*}
\begin{array}{ccccc}
  y_i\cdot v=0 & \text{ for } & 1\leq i\leq 2 & \text{ and } & v\in S^{\lambda}.
\end{array}
\end{align*}

By the PBW theorem for $\HH$ we have an isomorphism of $\CC$-vector space
\begin{align*}
\Delta(\lambda)\simeq \CC[x_1,x_2]\otimes_{\CC}S^{\lambda}.
\end{align*}

We want to describe the action of $\HH$ on the standard modules. As we say in the preliminars we have three kinds of $r$-partitions of 2: $\lambda_i$, $\lambda^i$ and $\lambda_{i,j}$.
The irreducible representations $S^{\lambda_i}$ and $S^{\lambda^i}$ are one dimensional with basis $v_{T}$. The irreducible representation $S^{\lambda_{i,j}}$ is two dimensional with basis $v_{T_1}$ and $v_{T_2}$.
The action of $W$ on the irreducible representations $S^{\lambda}$ is described in the following table

\begin{center}
\begin{tabular}{|l|l|}
  \hline
  \multicolumn{1}{|c|}{$\lambda_i$} & \multicolumn{1}{|c|}{$\lambda^i$} \\ \hline
  $\zeta_2\cdot v_T=\zeta^{i}v_T$ & $\zeta_2\cdot v_T=\zeta^{i}v_T$ \\
  $\zeta_1\cdot v_T=\zeta^{i}v_T$ & $\zeta_1\cdot v_T=\zeta^{i}v_T$ \\
  $I_2^t\cdot v_T=v_T$ & $I_2^t\cdot v_T=-v_T$ \\ \hline
  \multicolumn{2}{|c|}{$\lambda_{i,j}$}  \\ \hline
  \multicolumn{2}{|l|}{$\begin{array}{cc}
                                                                   \zeta_2\cdot v_{T_1}=\zeta^{j}v_{T_1} &  \hspace{9 mm}\zeta_2\cdot v_{T_2}=\zeta^{i}v_{T_2}
                                                                 \end{array}$} \\
  \multicolumn{2}{|l|}{$\begin{array}{cc}
                                                                    \zeta_1\cdot v_{T_1}=\zeta^{i}v_{T_1} & \hspace{9 mm}\zeta_1\cdot v_{T_2}=\zeta^{j}v_{T_2}
                                                                  \end{array}
                                  $} \\
  \multicolumn{2}{|l|}{$\begin{array}{cc}
                                                              I_2^t\cdot v_{T_1}=v_{T_2} & \hspace{13 mm}I_2^t\cdot v_{T_2}=v_{T_1}
                                                            \end{array}
                                  $} \\
  \hline
\end{tabular}
\end{center}
where $I_2$ is the identity matrix and $\zeta_i$ is the diagonal matrix that has $\zeta$ in the $i$-th position on the diagonal, and zero elsewhere.

\subsection{The action on $\Delta(\lambda)$}

The elements of $\Delta(\lambda)$ are sums of elements of the form $x_1^nx_2^m\otimes v_T$. Our interest is focus on how $\mathds{H}$ acts in elements of this form. The elements of $\CC[x_1,x_2]$ act by multiplication and the group elements act in the obvious way. Our interest is focus in how $y_1$ and $y_2$ act on an element $x_1^nx_2^m\otimes v_T$. 
In the following propositions the brackets over the sum ($[\ast]$) mean the entire part.

\begin{proposition} \label{actn1}
In $\lambda=\lambda_i$ the action of $y_1$ and $y_2$ in a generic $x_1^nx_2^m\otimes v_T$ is given by:
\begin{enumerate}[1)]
\item
$y_1\cdot x_1^nx_2^m\otimes v_T=$
$$\left\{\begin{array}{ccc}
                                   \left((n-d_i+d_{i-n}-c_0r)x_1^{n-1}x_2^m-\displaystyle c_0r\sum_{k=1}^{\left[\frac{n-m-1}{r}\right]}x_1^{n-kr-1}x_2^{m+kr}\right)\otimes v_T &  if  &n>m \\
                                   \left((n-d_i+d_{i-n})x_1^{n-1}x_2^m+\displaystyle c_0r\sum_{k=1}^{\left[\frac{m-n}{r}\right]}x_1^{n+kr-1}x_2^{m-kr}\right)\otimes v_T & if & n\leq m
                                                                   \end{array}
\right.$$
\item
$y_2\cdot x_1^nx_2^m\otimes v_T=$
$$\left\{\begin{array}{ccc}
                                   \left((m-d_i+d_{i-m})x_1^{n}x_2^{m-1}+\displaystyle c_0r\sum_{k=1}^{\left[\frac{n-m}{r}\right]}x_1^{n-kr}x_2^{m+kr-1}\right)\otimes v_T &  if  &n\geq m \\
                                   \left((m-d_i+d_{i-m}-c_0r)x_1^{n}x_2^{m-1}-\displaystyle c_0r\sum_{k=1}^{\left[\frac{m-n-1}{r}\right]}x_1^{n+kr}x_2^{m-kr-1}\right)\otimes v_T & if & n<m
                                                                    \end{array}
\right.$$
\end{enumerate}

\end{proposition}

\begin{proposition} \label{actn2}
In $\lambda=\lambda^i$ the action of $y_1$ and $y_2$ in a generic $x_1^nx_2^m\otimes v_T$ is given by:
\begin{enumerate}
\item
$y_1\cdot x_1^nx_2^m\otimes v_T=$
$$\left\{\begin{array}{ccc}
                                   \left((n-d_i+d_{i-n}+c_0r)x_1^{n-1}x_2^m+\displaystyle c_0r\sum_{k=1}^{\left[\frac{n-m-1}{r}\right]}x_1^{n-kr-1}x_2^{m+kr}\right)\otimes v_T &  if  &n>m \\
                                   \left((n-d_i+d_{i-n})x_1^{n-1}x_2^m-\displaystyle c_0r\sum_{k=1}^{\left[\frac{m-n}{r}\right]}x_1^{n+kr-1}x_2^{m-kr}\right)\otimes v_T & if & n\leq m
                                 \end{array}
\right.$$
\item
$y_2\cdot x_1^nx_2^m\otimes v_T=$
$$\left\{\begin{array}{ccc}
                                   \left((m-d_i+d_{i-m})x_1^{n}x_2^{m-1}-\displaystyle c_0r\sum_{k=1}^{\left[\frac{n-m}{r}\right]}x_1^{n-kr}x_2^{m+kr-1}\right)\otimes v_T &  if  &n\geq m \\
                                   \left((m-d_i+d_{i-m}+c_0r)x_1^{n}x_2^{m-1}+\displaystyle c_0r\sum_{k=1}^{\left[\frac{m-n-1}{r}\right]}x_1^{n+kr}x_2^{m-kr-1}\right)\otimes v_T & if & n<m
                                    \end{array}
\right.$$
\end{enumerate}

\end{proposition}


In $\lambda=\lambda_{i,j}$ we have two generators of $S^{\lambda}$, called $v_{T_1}$ and $v_{T_2}$.
\begin{proposition} \label{actn3}
When $\lambda=\lambda_{i,j}$ the action of $y_1$ and $y_2$ in a generic $x_1^nx_2^m\otimes v_{T_1}$ or a generic $x_1^nx_2^m\otimes v_{T_2}$ is given by:
\begin{enumerate}
\item
$y_1\cdot x_1^nx_2^m\otimes v_{T_1}=$

$$\left\{\begin{array}{ccc}
                                   (n-d_i+d_{i-n})x_1^{n-1}x_2^m\otimes v_{T_1}-rc_0\displaystyle \sum_{k=1}^{\left[ \frac{n-m-1+j-i}{r} \right]}x_1^{n-kr+j-i-1}x_2^{m+rk-j+i}\otimes v_{T_2} &  if  &n>m \\
                                   (n-d_i+d_{i-n})x_1^{n-1}x_2^m\otimes v_{T_1}+rc_0\displaystyle \sum_{k=0}^{\left[ \frac{m-n-j+i}{r} \right]}x_1^{n+kr+j-i-1}x_2^{m-rk-j+i}\otimes v_{T_2} & if & n<m \\
                                   (n-d_i+d_{i-n})x_1^{n-1}x_2^n\otimes v_{T_1} & if & n=m
                                 \end{array}
\right.$$

\item
$y_1\cdot x_1^nx_2^m\otimes v_{T_2}=$
$$\left\{\begin{array}{ccc}
                                   (n-d_j+d_{j-n})x_1^{n-1}x_2^m\otimes v_{T_2}-rc_0\displaystyle \sum_{k=0}^{\left[ \frac{n-m-1-j+i}{r} \right]}x_1^{n-kr-j+i-1}x_2^{m+rk+j-i}\otimes v_{T_1} &  if  &n>m \\
                                   (n-d_j+d_{j-n})x_1^{n-1}x_2^m\otimes v_{T_2}+rc_0\displaystyle \sum_{k=1}^{\left[ \frac{m-n+j-i}{r} \right]}x_1^{n+kr-j+i-1}x_2^{m-rk+j-i}\otimes v_{T_1} & if & n<m \\
                                   (n-d_j+d_{j-n})x_1^{n-1}x_2^n\otimes v_{T_2} & if & n=m
                                 \end{array}
\right.$$
\item
$y_2\cdot x_1^nx_2^m\otimes v_{T_1}=$
$$\left\{\begin{array}{ccc}
                                   (m-d_j+d_{j-m})x_1^{n}x_2^{m-1}\otimes v_{T_1}+rc_0\displaystyle \sum_{k=1}^{\left[ \frac{n-m-i+j}{r} \right]}x_1^{n-kr-i+j}x_2^{m+rk+i-j-1}\otimes v_{T_2} &  if  &n>m \\
                                   (m-d_j+d_{j-m})x_1^{n}x_2^{m-1}\otimes v_{T_1}-rc_0\displaystyle \sum_{k=0}^{\left[ \frac{m-n+i-j-1}{r} \right]}x_1^{n+kr-i+j}x_2^{m-rk+i-j-1}\otimes v_{T_2} & if & n<m \\
                                   (n-d_j+d_{j-n})x_1^{n-1}x_2^n\otimes v_{T_1} & if & n=m
                                 \end{array}
\right.$$
\item
$y_2\cdot x_1^nx_2^m\otimes v_{T_2}=$
$$\left\{\begin{array}{ccc}
                                   (m-d_i+d_{i-m})x_1^{n}x_2^{m-1}\otimes v_{T_2}+rc_0\displaystyle \sum_{k=0}^{\left[ \frac{n-m+i-j}{r} \right]}x_1^{n-kr+i-j}x_2^{m+rk-i+j-1}\otimes v_{T_1}  & if &n>m \\
                                   (m-d_i+d_{i-m})x_1^{n}x_2^{m-1}\otimes v_{T_2}-rc_0\displaystyle \sum_{k=1}^{\left[ \frac{m-n+j-i-1}{r} \right]}x_1^{n+kr+i-j}x_2^{m-rk-i+j-1}\otimes v_{T_1} & if & n<m \\
                                   (n-d_i+d_{i-n})x_1^{n-1}x_2^n\otimes v_{T_2} & if & n=m
                                 \end{array}
\right.$$
\end{enumerate}

\end{proposition}
\section{Singular polynomials}

\begin{definition}
A singular polynomial is an element $m\in \Delta(\lambda)$ such that $y_1\cdot m=0$ and $y_2\cdot m=0$.
\end{definition}
It is a consequence of the definition of the standard module $\Delta(V)$ that for any $\mathbb{H}$-module $M$ the map
$$\mathrm{Hom}_{\mathbb{H}}(\Delta(V),M)\xrightarrow{\sim} \mathrm{Hom}_{\CC W}(V,\mathrm{Sing}(M))$$
defined by
$$\phi\mapsto \phi|_{V}$$
is a bijection, where $\mathrm{Sing}(M)=\{m\in M | y\cdot m=0 \quad \forall y \in \mathfrak{h}\}$.
 For this reason we describe in the following subsections the singular polynomials in our three standard modules cases. These polynomials are described in \cite{SG1}, but we put it in an explicit form in order to deal with the poles that may appear.
\subsection{Case 1: $\lambda=\lambda_i$}
\begin{proposition} \label{sing1}
The following polynomials are singular polynomials in $\Delta(\lambda_i)$:

\begin{enumerate}
  \item $(x_1^r-x_2^r)^k\otimes v_t$  when  $c_0=\frac{k}{2}$ for positive odd $k$.
  \item $x_1^nx_2^n\otimes v_t$  when  $n-d_i+d_{i-n}=0. $
  \item For $kr<n< (k+1)r$ , $\displaystyle\alpha_l={k \choose l}$ and $\displaystyle\beta_l=\frac{c_0(c_0-1)...(c_0-l)}{(c_0-k)(c_0-(k-1))...(c_0-(k-l))}$ \\
  $p(x_1,x_2)=x_1^n+\displaystyle\sum_{l=0}^{\left[\frac{k}{2}\right]}\alpha_l\beta_lx_1^{n-(k-l)r}x_2^{(k-l)r}+\sum_{l=1}^{\left[\frac{k-1}{2}\right]}\alpha_l\beta_{l-1}x_1^{n-lr}x_2^{lr} $\\
  when  $n-d_i+d_{i-n}-c_0r=0$  (if $c_0=m$ is an integer that indeterminates some $\beta_l$, then the polynomial is $(c_0-m)p(x_1,x_2)$).

\end{enumerate}
\end{proposition}
\begin{proof}
The fact that the polynomial in $(a)$ is a singular polynomial is Proposition 5.2 in \cite{Dun3}. Using our formulas and the fact  that $n-d_i+d_{i-n}=0$ we have:
$$y_1\cdot x_1^nx_2^n=(n-d_i+d_{i-n})x_1^{n-1}x_2^n=0$$
$$y_2\cdot x_1^nx_2^n=(n-d_i+d_{i-n})x_1^{n}x_2^{n-1}=0$$
this proves that $(b)$ is a singular polynomial.
To prove the that $(c)$ is a singular polynomial we construct a table that shows how $y_1$ and $y_2$ act in the monomials that appear in the polynomial. These tables correspond to the matrix system for the action of $y_1$ in the polynomial. The size of the tables depends on $k$. The tables are different if $k$ is even or odd. The following tables for the action of $y_1$ when $k=5$ and $k=6$ give an idea of how to construct a table in general.
\begin{center}
\renewcommand{\arraystretch}{1.5}
\begin{tabular}{|l||c|c|c|c|c|c|c|}
  \hline
   $k=5$  &  $x_1^{n-r-1}x_2^r$ & $x_1^{n-2r-1}x_2^{2r}$ & $x_1^{n-3r-1}x_2^{3r}$ & $x_1^{n-4r-1}x_2^{4r}$ & $x_1^{n-5r-1}x_2^{5r}$  \\ \hline \hline
  $x_1^n$ &       $-c_0r$     &  $-c_0r$     &  $-c_0r$     &   $-c_0r$    &  $-c_0r$          \\ \hline
  $x_1^{n-r}x_2^r$ &       $-r\alpha_1\beta_0$   &   $-c_0r\alpha_1\beta_0$    &  $-c_0r\alpha_1\beta_0$     &   $-c_0r\alpha_1\beta_0$    &     0      \\ \hline
  $x_1^{n-2r}x_2^{2r}$ &        0    &  $-2r\alpha_2\beta_1$    &    $-c_0r\alpha_2\beta_1$   &  0    &     0      \\ \hline
  $x_1^{n-3r}x_2^{3r}$ &       0     &     0 &   $(c_0-3)r\alpha_2\beta_2$   &    0  &         0  \\ \hline
  $x_1^{n-4r}x_2^{4r}$ &       0     & $c_0r\alpha_1\beta_1$     &   $c_0r\alpha_1\beta_1$   &  $(c_0-4)r\alpha_1\beta_1$   &        0   \\ \hline
  $x_1^{n-5r}x_2^{5r}$ &$c_0r\alpha_0\beta_0$&  $c_0r\alpha_0\beta_0$    &    $c_0r\alpha_0\beta_0$  &    $c_0r\alpha_0\beta_0$  &     $(c_0-5)r\alpha_0\beta_0$      \\
  \hline
\end{tabular}
\end{center}
\begin{center}
\renewcommand{\arraystretch}{1.5}
\begin{tabular}{|l||c|c|c|c|c|c|c|c|}
  \hline
   $k=6$  &  $x_1^{n-r-1}x_2^{r}$ & $x_1^{n-2r-1}x_2^{2r}$ & $x_1^{n-3r-1}x_2^{3r}$ & $x_1^{n-4r-1}x_2^{4r}$ & $x_1^{n-5r-1}x_2^{5r}$ & $x_1^{n-6r-1}x_2^{6r}$ \\ \hline \hline
  $x_1^{n}$ &       $-c_0r$     &  $-c_0r$     &  $-c_0r$     &   $-c_0r$    &  $-c_0r$  &     $-c_0r$   \\ \hline
  $x_1^{n-r}x_2^{r}$ & $-r\alpha_1\beta_0$ & $-c_0r\alpha_1\beta_0$    &  $-c_0r\alpha_1\beta_0$     &   $-c_0r\alpha_1\beta_0$    &     $-c_0r\alpha_1\beta_0$   & 0 \\ \hline
  $x_1^{n-2r}x_2^{2r}$ &        0    &  $-2r\alpha_2\beta_1$    &    $-c_0r\alpha_2\beta_1$   &  $-c_0r\alpha_2\beta_1$    &     0    & 0 \\ \hline
  $x_1^{n-3r}x_2^{3r}$ &       0     &     0 &   $-3r\alpha_3\beta_3$   &    0  &         0  & 0\\ \hline
  $x_1^{n-4r}x_2^{4r}$ &       0     & 0 & $c_0r\alpha_2\beta_2$     &    $(c_0-4)r\alpha_2\beta_2$   &        0  &0 \\ \hline
  $x_1^{n-5r}x_2^{5r}$ & 0&  $c_0r\alpha_1\beta_1$   &  $c_0\alpha_1\beta_1$    &    $c_0r\alpha_1\beta_1$  &       $(c_0-5)r\alpha_1\beta_1$ & 0    \\ \hline
  $x_1^{n-6r}x_2^{6r}$ & $c_0r\alpha_0\beta_0$ & $c_0r\alpha_0\beta_0$ & $c_0r\alpha_0\beta_0$ & $c_0r\alpha_0\beta_0$ & $c_0r\alpha_0\beta_0$&$(c_0-6)r\alpha_0\beta_0$\\
  \hline
\end{tabular}
\end{center}
We need to prove that the columns of the table add up zero. The prove works for a table of any size. First we prove that, if the $i$ column add up zero, then the $k-i+1$ column will add up zero. For this, these two columns involved only differ in the factors of the middle. In the $k-i+1$ column we only have a $(c_0-(k-i+1))r\alpha_{i-1}\beta_{i-1}$ and in the $i$ column we have $-ir\alpha_{i}\beta_{i-1}$ and $c_0r\alpha_{i-1}\beta_{i-1}$. Considering this we only need to prove that
$$(c_0-(k-i+1))\alpha_{i-1}\beta_{i-1}r=c_0r\alpha_{i-1}\beta_{i-1}-ir\alpha_{i}\beta_{i-1},$$ and this is only true if $$(k-i+1)\alpha_{i-1}=i\alpha_{i}$$
which is clearly true if we consider the definition of $\alpha_l$.

Now we prove that the $k-i+1$ column, for $1\leq i \leq \left[\frac{k+1}{2}\right]$, add up zero. The sum of these columns is
$$-c_0r-\displaystyle\sum_{l=1}^{i-1} c_0r\alpha_l\beta_{l-1}+\displaystyle\sum_{l=1}^{i-1}c_0r\alpha_{l-1}\beta_{l-1}+(c_0-(k-i+1))r\alpha_n\beta_n$$
for $i=1,2,...,\left[\frac{k+1}{2}\right]$. If we rewrite this last expression, we need to prove that $$-c_0\left(1+\displaystyle\sum_{l=1}^{i-1}(\alpha_l-\alpha_{l-1})\beta_{l-1}\right)+(c_0-(k-i+1))\alpha_{i-1}\beta_{i-1}=0.$$

We proceed by induction. For $i=1$ we have $-c_0+(c_0-k)\alpha_0\beta_0$=0, and considering the definition of $\alpha_0$
 and $\beta_0$ we get $$-c_0+(c_0-k)\frac{c_0}{c_0-k}=0.$$ Now assuming it works for $i$ we need to prove that
 $$-c_0\left(1+\displaystyle\sum_{l=1}^{i}(\alpha_l-\alpha_{l-1})\beta_{l-1}\right)+(c_0-(k-i))\alpha_{i}\beta_{i}=0.$$
In order to prove this, we have:

$\begin{array}{cl}
   & -c_0\left(1+\displaystyle\sum_{l=1}^{i}(\alpha_l-\alpha_{l-1})\beta_{l-1}\right)+(c_0-(k-i))\alpha_{i}\beta_{i} \\
  = & -c_0\left(1+\displaystyle\sum_{l=1}^{i-1}(\alpha_l-\alpha_{l-1})\beta_{l-1}+(\alpha_{i}-\alpha_{i-1})\beta_{i-1}\right)+(c_0-(k-i))\alpha_{i}\beta_{i} \\
  = & \cancel{-c_0\left(1+\displaystyle\sum_{l=1}^{i-1}(\alpha_l-\alpha_{l-1})\beta_{l-1}\right)}+\cancel{(c_0-(k-i+1))\alpha_{i-1}\beta_{i-1}} \\
    & -(c_0-(k-i+1))\alpha_{i-1}\beta_{i-1}-c_0(\alpha_{i}-\alpha_{i-1})\beta_{i-1}+(c_0-(k-i))\alpha_{i}\beta_{i} \\
  = & -c_0\alpha_{i-1}\beta_{i-1}+(k-i+1)\alpha_{i-1}\beta_{i-1}-c_0\left(\displaystyle\frac{k-i+1}{i}\alpha_{i-1}-\alpha_{i-1}\right)\beta_{i-1}\\
    & +\cancel{(c_0-k+i)}\displaystyle  \frac{k-i+1}{i}\alpha_{i-1}\frac{c_0-i}{\cancel{c_0-k+i}}\beta_{i-1}\\
  = & \left(-c_0+k-i+1-c_0\displaystyle\frac{k-i+1}{i}+c_0+\displaystyle\frac{(k-i+1)(c_0-i)}{i}\right)\alpha_{i-1}\beta_{i-1}\\
  = & \displaystyle\frac{(k-i+1)(i)-c_0(k-i+1)+(k-i+1)(c_0-i)}{i}\alpha_{i-1}\beta_{i-1}\\
  = & 0.
\end{array}$
\newline
We used the induction hypothesis and the fact that
 \begin{center}
 $\alpha_{i}=\displaystyle\frac{k-i+1}{i}\alpha_{i-1}$ and $\beta_{i}=\displaystyle\frac{c_0-i}{c_0-k+i}\beta_{i-1}$.
 \end{center}

  In addition, if $c_0=m$ and $c_0-m$ indeterminate some $\beta_l$, then the polynomial
 that we are looking for is $(c_0-m)p(x_1,x_2)$. This new polynomial works, because the factor $(c_0-m)$ appears almost in degree one in the denominator of
 some coefficients. Now we need to prove that $y_2$ annihilate the polynomial too, but the corresponding table for $y_2$ shows us the same system to be solved as before.

\end{proof}

\subsection{Case 2: $\lambda=\lambda^i$}
\begin{proposition}\label{sing2}
The following polynomials are singular polynomials in $\Delta(\lambda^i)$:

\begin{enumerate}
  \item $(x_1^r-x_2^r)^k\otimes v_t$  when  $c_0=-\frac{k}{2}$ for positive odd $k$.
  \item $x_1^nx_2^n\otimes v_t$  when  $n-d_i+d_{i-n}=0 $
  \item For $kr<n<(k+1)r$ , $\displaystyle\alpha_l={k \choose l}$ and $\displaystyle\beta_l=\frac{c_0(c_0+1)...(c_0+l)}{(c_0+k)(c_0+(k-1))...(c_0+(k-l))}$
  $$p(x_1,x_2)=x_1^n+\displaystyle\sum_{l=0}^{\left[\frac{k}{2}\right]}\alpha_l\beta_lx_1^{n-(k-l)r}x_2^{(k-l)r}+\sum_{l=1}^{\left[\frac{k-1}{2}\right]}\alpha_l\beta_{l-1}x_1^{n-lr}x_2^{lr} $$
  when  $n-d_i+d_{i-n}+c_0r=0$  (if $c_0=-m$ is an integer that indeterminates some $\beta_l$, then the polynomial is $(c_0+m)p(x_1,x_2)$).

\end{enumerate}
\end{proposition}

\begin{proof}
The proof in this case is as in the $\lambda_i$ case. We only need to change $c_0$ into $-c_0$.
\end{proof}

\subsection{Case 3: $\lambda=\lambda_{i,j}$}
\begin{proposition} \label{sing3}
For $\lambda_{i,j}$ and $i<j$
\begin{enumerate}[1)]
\item  We have the two following singular polynomials
\begin{enumerate}
          \item If $kr<n+j-i< (k+1)r$, $n-d_i+d_{i-n}=0$ , $s_t=j-i-d_j+d_i-tr$.
              $$p(x_1,x_2)=\left(x_1^n+\displaystyle\sum_{l=1}^{k-1}b_lx_1^{n-lr}x_2^{lr}\right) \otimes v_{T_1}+\displaystyle\sum_{l=1}^ka_l x_1^{n-lr+j-i}x_2^{lr-j+i}\otimes v_{T_2}$$

          \item If $(k-1)r<n+i-j < kr$, $n-d_j+d_{j-n}=0$ , $s_t=i-j-d_i+d_j-(t-1)r$.
              $$p(x_1,x_2)=\left(x_2^n+\displaystyle\sum_{l=1}^{k-1}b_lx_1^{lr}x_2^{n-lr}\right) \otimes v_{T_1}+\displaystyle\sum_{l=0}^{k-1}a_{l+1} x_1^{lr+j-i}x_2^{n-lr-j+i}\otimes v_{T_2}$$
             \end{enumerate}
Where the coefficients satisfy the recursive system
\begin{center}\begin{itemize}
               \item $s_1a_1=c_0r$
               \item $s_la_l=s_{k-l+1}a_{k-l+1}$ for $1\leq l< \left[ \frac{k+1}{2}\right]$
               \item $lb_l=(k-l)b_{k-l}$ for $1\leq l< \left[ \frac{k+1}{2}\right]$
               \item $a_l=\frac{c_0r}{s_l}\left(\displaystyle\sum_{j=1}^{l-1}\frac{k-2j}{j}b_{k-j}+1\right)$
               \item $b_l=\frac{c_0}{l}\left(\displaystyle\sum_{j=0}^{l-1}\left(\frac{(k-2j-1)r}{s_{k-j}}\right)a_{j+1}\right)$
             \end{itemize}\end{center}
For this polynomials, if $s_t=0$ for some $t$, then the polynomial is $s_t\cdot p(x_1,x_2)$.
\item
For $1\leq l \leq \left[ \frac{k+1}{2}\right]$ define
\begin{center}\begin{itemize}
              \item $a_l=\frac{1}{l!}\displaystyle \frac{c_0(c_0-1)...(c_0-(l-1))k(k-1)...(k-(l-1)}{(c_0-k)(c_0-(k-1))...(c_0-(k-(l-1))}$
              \item $a_{k-l}=\displaystyle \frac{1}{l!}\frac{c_0(c_0-1)(c_0-2)...(c_0-l)k(k-1)...(k-(l-1))}{(c_0-k)(c_0-(k-1))...(c_0-(k-l))}$
              \item $a_k=\displaystyle \frac{c_0}{c_0-k}$
              \end{itemize}
              \end{center}
We have the two following singular polynomials
\begin{enumerate}
             \item If $n=i-j+(k+1)r$ , $n=d_i-d_j+rc_0$.\\
              $p(x_1,x_2)=(x_1^n\otimes v_{T_1}-x_2^n\otimes v_{T_2})+\displaystyle\sum_{l=1}^{k}a_l\left(x_1^{n-rl}x_2^{rl}\otimes v_{T_1}-x_1^{rl}x_2^{n-rl}\otimes v_{T_2} \right)$
             \item If $n=j-i+kr$ , $n=d_j-d_i+rc_0$.\\
             $p(x_1,x_2)=(x_1^n\otimes v_{T_2}-x_2^n\otimes v_{T_1})+\displaystyle\sum_{l=1}^{k}a_l\left(x_1^{n-rl}x_2^{rl}\otimes v_{T_2}-x_1^{rl}x_2^{n-rl}\otimes v_{T_1} \right)$
            \end{enumerate}

     \item    For $1\leq l \leq \left[ \frac{k+1}{2}\right]$ define
\begin{center}\begin{itemize}
             \item $a_l=\frac{1}{l!}\displaystyle \frac{c_0(c_0+1)...(c_0+(l-1))k(k-1)...(k-(l-1)}{(c_0+k)(c_0+(k-1))...(c_0+(k-(l-1))}$
             \item $a_{k-l}=\displaystyle \frac{1}{l!}\frac{c_0(c_0+1)(c_0+2)...(c_0+l)k(k-1)...(k-(l-1))}{(c_0+k)(c_0+(k-1))...(c_0+(k-l))}$
             \item $a_k=\displaystyle \frac{c_0}{c_0+k}$
           \end{itemize}\end{center}
We have the two following singular polynomials
\begin{enumerate}
          \item If $n=i-j+(k+1)r$ , $n=d_i-d_j-rc_0$. \\
          $p(x_1,x_2)=(x_1^n\otimes v_{T_1}+x_2^n\otimes v_{T_2})+\displaystyle\sum_{l=1}^{k}a_l\left(x_1^{n-rl}x_2^{rl}\otimes v_{T_1}+x_1^{rl}x_2^{n-rl}\otimes v_{T_2} \right)$
           \item If $n=j-i+kr$ , $n=d_j-d_i-rc_0$.\\
           $p(x_1,x_2)=(x_1^n\otimes v_{T_2}+x_2^n\otimes v_{T_1})+\displaystyle\sum_{l=1}^{k}a_l\left(x_1^{n-rl}x_2^{rl}\otimes v_{T_2}+x_1^{rl}x_2^{n-rl}\otimes v_{T_1} \right)$
           \end{enumerate}
\end{enumerate}
 For the polynomials $(2.a)$, $(2.b)$ $(3.a)$ and $(3.b)$, if $k$ is an even number we compute $a_{\frac{k}{2}}$ considering the definition of $a_l$ instead the definition of $a_{k-l}$. If $c_0$ is an integer $m$ such that the denominator of some $a_l$ is zero, then the polynomials are $(c_0+m)\cdot p(x_1,x_2)$ or $(c_0-m)\cdot p(x_1,x_2)$.

\end{proposition}
\begin{proof}
As in the proof for $\lambda_i$ we construct tables for a specific value of $k$ in order to understand the action over the monomials involved, but the prove is given for any value of $k$.
\begin{enumerate}[1)]
  \item First we give the tables for $k=4$ and $k=5$. The following table is valid for the polynomial of case $(1.a)$. Define $N=n+j-i$ and for each monomial we write $x_1^m$ instead of $x_1^mx_2^{n-m}$.
  \begin{center}
\begin{adjustbox}{max width=\textwidth}
\newcolumntype{b}{>{\columncolor{Gray}}c}
\newcolumntype{r}{>{\columncolor{white}}c}
\begin{tabular}{|c||b|b|b||r|r|r|r|}
  \hline
  $k=4$ & $x_1^{n-r-1}$ & $x_1^{n-2r-1}$ & $x_1^{n-3r-1}$  & $x_1^{N-r-1}$ & $x_1^{N-2r-1}$ & $x_1^{N-3r-1}$ & $x_1^{N-4r-1}$  \\ \hline \hline
  \cellcolor{Gray} $x_1^n$ & 0 & 0 & 0  & $-c_0r$ & $-c_0r$ & $-c_0r$ & $-c_0r$   \\ \hline
  \cellcolor{Gray}$x_1^{n-r}$ & $-rb_1$ & 0  & 0 & 0 & $-c_0rb_1$  & $-c_0rb_1$ & 0  \\ \hline
  \cellcolor{Gray}$x_1^{n-2r}$ & 0 & $-2rb_2$ & 0  & 0 & 0  & 0 & 0  \\ \hline
  \cellcolor{Gray}$x_1^{n-3r}$ & 0 & 0 & $-3rb_3$  & 0  & $c_0rb_3$ & $c_0rb_3$ & 0 \\ \hline \hline
  \cellcolor{white}$x_1^{N-r}$ & $-c_0ra_1$ & $-c_0ra_1$ & $-c_0ra_1$  & $s_1a_1$ & 0 & 0 & 0   \\ \hline
  \cellcolor{white}$x_1^{N-2r}$ & 0 & $-c_0ra_2$ &  0 & 0 & $s_2a_2$ & 0 & 0   \\ \hline
  \cellcolor{white}$x_1^{N-3r}$ & 0 & $c_0ra_3$ & 0 & 0  & 0 & $s_3a_3$ & 0   \\ \hline
  \cellcolor{white}$x_1^{N-4r}$ & $c_0ra_4$ & $c_0ra_4$ & $c_0ra_4$  & 0 & 0 & 0 & $s_4a_4$   \\
  \hline
\end{tabular}
\end{adjustbox}
\end{center}

\begin{center}
\newcolumntype{b}{>{\columncolor{Gray}}c}
\newcolumntype{r}{>{\columncolor{white}}c}
\begin{adjustbox}{width=0.9\textwidth}
\begin{tabular}{|c||b|b|b|b||r|r|r|r|r|}
  \hline
  $k=5$  & $x_1^{n-1-r}$ & $x_1^{n-1-2r}$ & $x_1^{n-1-3r}$ & $x_1^{n-1-4r}$ & $x_1^{N-r-1}$ & $x_1^{N-2r-1}$ & $x_1^{N-3r-1}$ & $x_1^{N-4r-1}$ & $x_1^{N-5r-1}$ \\ \hline \hline
  \cellcolor{Gray} $x_1^{n}$ & 0 & 0 & 0 & 0 & $-c_0r$ & $-c_0r$ & $-c_0r$ & $-c_0r$ & $-c_0r$  \\ \hline
  \cellcolor{Gray}$x_1^{n-r}$ & $-rb_1$ & 0 & 0 & 0 & 0 & $-c_0rb_1$ & $-c_0rb_1$ & $-c_0rb_1$ & 0  \\ \hline
  \cellcolor{Gray}$x_1^{n-2r}$ & 0 & $-2rb_2$ & 0 & 0 & 0 & 0 & $-c_0rb_2$ & 0 & 0  \\ \hline
  \cellcolor{Gray}$x_1^{n-3r}$ & 0 & 0 & $-3rb_3$ & 0 & 0 & 0 & $c_0rb_3$ & 0 & 0 \\ \hline
  \cellcolor{Gray}$x_1^{n-4r}$ & 0 & 0 & 0 & $-4rb_4$ & 0 & $c_0rb_4$ & $c_0rb_4$ & $c_0rb_4$ & 0  \\ \hline \hline
  \cellcolor{white}$x_1^{N-r}$ & $-c_0ra_1$ & $-c_0ra_1$ & $-c_0ra_1$ & $-c_0ra_1$ & $s_1a_1$ & 0 & 0 & 0 & 0  \\ \hline
  \cellcolor{white}$x_1^{N-2r}$ & 0 & $-c_0ra_2$ & $-c_0ra_2$ & 0 & 0 & $s_2a_2$ & 0 & 0 & 0  \\ \hline
  \cellcolor{white}$x_1^{N-3r}$ & 0 & 0 & 0 & 0 & 0 & 0 & $s_3a_3$ & 0 & 0  \\ \hline
  \cellcolor{white}$x_1^{N-4r}$ & 0 & $c_0ra_4$ & $c_0ra_4$ & 0 & 0 & 0 & 0 & $s_4a_4$ & 0  \\ \hline
  \cellcolor{white}$x_1^{N-5r}$ & $c_0ra_5$ & $c_0ra_5$ & $c_0ra_5$ & $c_0ra_5$ & 0 & 0 & 0 & 0 & $s_5a_5$  \\
  \hline
\end{tabular}
\end{adjustbox}
\end{center}
The gray part in these tables means that the monomial is $\otimes_{v_{T_1}}$ and the white part means $\otimes_{v_{T_2}}$. Now we prove that each column add up zero if the $a_l$ and $b_l$ satisfy the system involved. The first column of the white part says that $a_1s_1=c_0r$ ,which is the first condition of our system. Now if we look only at the white part we can see that the $l$ column and the $k-l+1$ column have the same first $k$ entries. In the other entries  we have $a_ls_l$ in the $l$ column and $a_{k-l+1}s_{k-l+1}$ in the $k-l+1$ column. This implies that $a_ls_l=a_{k-l+1}s_{k-l+1}$
and this is the second part of the system. If we look at the gray part we can see that the last $k-1$ entries are the same in the $l$ column and in the $k-l$ column. We can also see that the first $k-1$ entries of these columns are $-lrb_l$ in the $l$ column and $-(k-l)rb_{k-l}$ in the $k-l$ column. This implies that $lb_l=(k-l)b_{k-l}$ which, is the third part of the system. For the fourth part we have to look  at the white part of the table. We have:
 $$a_ls_l=c_0r+\sum_{j=1}^{l-1}c_0rb_j-c_0rb_{k-j}$$ and if we combine this with $lb_l=(k-l)b_{k-l}$, we get
  $$a_ls_l=c_0r+\sum_{j=1}^{l-1}c_0r\frac{k-j}{j}b_{k-j}-c_0rb_{k-j}=c_0r\sum_{j=1}^{l-1}\frac{k-2j}{j}b_{k-j}+1.$$ This implies the fourth part of the system.
For the fifth part we have to look at the gray part of the table to get $$lrb_l=\sum_{j=0}^{l-1}-c_0ra_{j+1}+c_0ra_{k-j}$$ and  we can use $a_ls_l=a_{k-l+1}s_{k-l+1}$ to get
$$lrb_l=\sum_{j=0}^{l-1}-c_0r\frac{s_{k-j}}{s_{j+1}}a_{k-j}+c_0ra_{k-j}=c_0r\sum_{j=0}^{l-1}\frac{s_{k-j}-s_{j+1}}{s_{k-j}}a_{k-j}.$$
Finally we have
$s_{k-s}-s_{j+1}=(k-2j-1)r$ and this completes the last part of the system.

The table for $y_2$ represent the same system.

For the polynomial $(1.b)$ the tables represents the same system.
  \item Considering the monomials of the first singular polynomial and writing $x_1^m$ instead of $x_1^mx_2^{n-m}$ the table for $k=4$ is:

  \begin{center}
\begin{adjustbox}{max width=\textwidth}
\newcolumntype{b}{>{\columncolor{Gray}}c}
\newcolumntype{r}{>{\columncolor{white}}c}
\begin{tabular}{|c||b|b|b|b|b||r|r|r|r|}
  \hline
  $k=4$  & $x_1^{n-1}$ & {$x_1^{n-1-r}$} & {$x_1^{n-1-2r}$} & {$x_1^{n-1-3r}$} & $x_1^{n-1-4r}$ & $x_1^{r-1}$ & $x_1^{2r-1}$ & $x_1^{3r-1}$ & $x_1^{4r-1}$ \\ \hline \hline
  \cellcolor{Gray}{$x_1^{n}$} & $c_0r$ & 0 & 0 & 0 & 0 & $-c_0r$ & $-c_0r$ & $-c_0r$ & $-c_0r$  \\ \hline
  \cellcolor{Gray}{$x_1^{n-r}$} & 0 & \resizebox{1.5cm}{!}{$a_1(c_0r-r)$} & 0 & 0 & 0 & 0 & $-a_1c_0r$ & $-a_1c_0r$ & 0  \\ \hline
  \cellcolor{Gray}{$x_1^{n-2r}$} & 0 & 0 & \resizebox{1.5cm}{!}{$a_2(c_0r-2r)$} & 0 & 0 & 0 & 0 & 0 & 0  \\ \hline
  \cellcolor{Gray}$x_1^{n-3r}$ & 0 & 0 & 0 & \resizebox{1.5cm}{!}{$a_3(c_0r-3r)$} & 0 & 0 & $a_3c_0r$ & $a_3c_0r$ & 0 \\ \hline
  \cellcolor{Gray}$x_1^{n-4r}$& 0 & 0 & 0 & 0 & \resizebox{1.5cm}{!}{$a_4(c_0r-4r)$} & $a_4c_0r$ & $a_4c_0r$ & $a_4c_0r$ & $a_4c_0r$  \\ \hline \hline
  \cellcolor{white}$x_1^{0}$& $-c_0r$ & $-c_0r$ & $-c_0r$ & $-c_0r$ & $-c_0r$ & 0 & 0 & 0 & 0  \\ \hline
  \cellcolor{white}$x_1^{r}$ & 0 & $-a_1c_0r$ & $-a_1c_0r$ & $-a_1c_0r$ & 0 & $-a_1r$ & 0 & 0 & 0  \\ \hline
  \cellcolor{white}$x_1^{2r}$ & 0 & 0 & $-a_2c_0r$ & 0 & 0 & 0 & $-a_22r$ & 0 & 0  \\ \hline
  \cellcolor{white}$x_1^{3r}$ & 0 & 0 & $a_3c_0r$ & 0 & 0 & 0 & 0 & $-a_33r$ & 0  \\ \hline
  \cellcolor{white}$x_1^{4r}$ & 0 & $a_4c_0r$ & $a_4c_0r$ & $a_4c_0r$ & 0 & 0 & 0 & 0 & $-a_44r$  \\
  \hline
\end{tabular}
\end{adjustbox}
\end{center}

We prove that the columns add up zero. We can see that the first column adds up zero. The sum of the other columns of the gray part
 is exactly the same sum of the columns of the white part. Therefore  we prove that the white columns add up zero. To prove this fact for a generic table, we proceed by induction. Firstly the last column says that $-c_0r+a_kc_0r-a_kkr=0$. This implies that $a_k=\frac{c_0}{c_0-k}$.
   We need to prove that the formulas work for $l=1$. For this we need to look at the first column of the white part, which says that $-c_0r+a_kc_0r-a_1r=0$.
   Replacing the $a_k$-term we get that $-c_0r+\frac{c_0}{c_0-k}c_0r-a_1r=0$, which implies that $a_1=\frac{c_0k}{c_0-k}$
   and it coincides with the formulas. The next step is to prove that the $k-1$ column of the white part adds up zero. This column says that $$-c_0r-a_1c_0r+a_{k-1}c_0r+a_kc_0r-a_{k-1}(k-1)r=0$$  and if we replace $a_k$ and $a_1$ we get $$-c_0r-c_0r\frac{c_0k}{c_0-k}+a_{k-1}c_0r+\frac{c_0}{c_0-k}c_0r-a_{k-1}(k-1)r=0.$$ This implies that $a_{k-1}=\frac{c_0(c_0-1)k}{(c_0-k)(c_0-(k-1))}.$ This proves the case when $l=1$.

   Assuming that the formula works for $n$, the corresponding sum of the column to compute $a_{n+1}$ is $$-c_0r-\sum_{j=1}^{n}c_0ra_j+\sum_{j=0}^{n}c_0ra_{k-j}-a_{n+1}(n+1)r=0.$$
    This implies that $$\begin{array}{cl}
                                a_{n+1}(n+1)r & =\displaystyle-c_0r-\sum_{j=1}^{n}c_0ra_j+\sum_{j=0}^{n}c_0ra_{k-j} \\
                                 & =\displaystyle-c_0r\left(1+\sum_{j=1}^n(a_j-a_{k-j})-a_k\right) \\
                                 & =\displaystyle-c_0r\left(1-\frac{c_0}{c_0-k}+\sum_{j=1}^na_j\frac{2j-k}{c_0-(k-j)}\right).
                              \end{array}.  $$
If we prove that $$-c_0\sum_{j=1}^na_j\frac{2j-k}{c_0-(k-j)}=a_{n+1}(n+1)+c_0-\frac{c_0^2}{c_0-k}$$ we have proven the formula.
 We will prove this last claim by induction. If $n=1$ we have $$-c_0a_1\frac{2-k}{c_0-(k-1)}=2a_{2}+c_0-\frac{c_0^2}{c_0-k}.$$
  We can see that this is true by replacing $a_1=\frac{c_0k}{c_0-1}$ and $a_2=\frac{1}{2}\frac{c_0(c_0-1)k(k-1)}{(c_0-k)(c_0-(k-1))}$.
   For the induction step we assume that $$-c_0\sum_{j=1}^na_j\frac{2j-k}{c_0-(k-j)}=a_{n+1}(n+1)+c_0-\frac{c_0^2}{c_0-k}$$ is true and we need to prove that
   $$-c_0\sum_{j=1}^{n+1}a_j\frac{2j-k}{c_0-(k-j)}=a_{n+2}(n+2)+c_0-\frac{c_0^2}{c_0-k}$$ is also true. Starting with $$-c_0\sum_{j=1}^{n+1}a_j\frac{2j-k}{c_0-(k-j)}=
   -c_0\sum_{j=1}^{n}a_j\frac{2j-k}{c_0-(k-j)}-c_0a_{n+1}\frac{2(n+1)-k}{c_0-(k-(n+1))}$$ and considering the induction hypothesis we get
    $$-c_0\sum_{j=1}^{n+1}a_j\frac{2j-k}{c_0-(k-j)}=a_{n+1}(n+1)+c_0-\frac{c_0^2}{c_0-k}-c_0a_{n+1}\frac{2(n+1)-k}{c_0-(k-(n+1))}.$$
     This last equation is true because  $$a_{n+1}(n+1)-c_0a_{n+1}\frac{2(n+1)-k}{c_0-(k-(n+1))}=
      a_{n+1}\frac{(c_0-(n+1))(k-(n+1))}{c_0-(k-(n+1)}=(n+2)a_{n+2}$$  (the las equality is by the definition of $a_{n+1}$
      comparing with $a_{n+2}$) and the proof is complete.

      For the polynomials $(2.b)$ the tabes involve the same system to solve.
  \item For the polynomials $(3.a)$ and $(3.b)$ the proof is the same as before. We only need to change $c_0$ into $-c_0$.
\end{enumerate}

\end{proof}

We give an example to show how compute the polynomial $(1.a)$.

\begin{example}
Suppose that we have the following data:
\begin{multicols}{2}
\begin{itemize}
  \item $r=4$
  \item $d_0=13$
  \item $d_1=-13$
  \item $d_2=0$
  \item $d_3=0$
  \item $c_0=-3$
\end{itemize}
\end{multicols}

        If we consider $n=13$ for $\lambda_{0,1}$ we have that $$13-d_0+d_{0-13}=13-13-0=0$$ and $$12<13+1-0<16$$ thus $k=3$. This applies to the polynomials $(1.a)$. The polynomial annihilated is:
        $$p(x_1,x_2)=\left(x_1^{13}+b_1x_1^9x_2^4+b_2x_1^5x_2^8\right)\otimes v_{T_1}+(a_1x_1^{10}x_2^3+a_2x_1^6x_2^7+a_3x_1^{2}x_2^{11})\otimes v_{T_2}.$$
        We need to compute the coefficients. In this case
        $$\begin{array}{ccc}
            s_1=23, & s_2=19, & s_3=15
          \end{array}.
        $$
        $s_1a_1=c_0r$ implies that $$a_1=\displaystyle-\frac{12}{23}.$$ Using the second part of the system that says $s_1a_1=s_3a_3$, we have that $$a_3=\displaystyle-\frac{4}{5}.$$ If we compute $b_1$ using the last part of the system,  $b_1=c_0\left(\frac{2r}{s_3}\right)a_1$ and this implies that $$b_1=\displaystyle\frac{96}{115}.$$ Using part three, we have that $b_1=2b_2$. And this implies $$b_2=\displaystyle \frac{48}{115}.$$
        We finish computing $a_2$. $$a_2=\frac{c_0r}{s_2}(b_2+1)=-\frac{1956}{2185}$$ this way we compute all the coefficients. The polynomial is:
        $$\displaystyle p(x_1,x_2)=\left(x_1^{13}+\frac{96}{115}x_1^9x_2^4+\frac{48}{115}x_1^5x_2^8\right)\otimes v_{T_1}-\left(\frac{12}{23}x_1^{10}x_2^3+
        \frac{1956}{2185}x_1^6x_2^7+\frac{4}{5}x_1^{2}x_2^{11}\right)\otimes v_{T_2}.$$
\end{example}

\section{Main theorems}

  If we have an $r$-partition $\lambda=(\lambda^0,\lambda^1,...,\lambda^{r-1})$, define the \textit{content} of a box $b \in \lambda^{i}$ by $j-k$, if $b$ is in the $k$ row and in the $j$ column from $\lambda^{i}$. We write it $ct(b)$ = \textit{content} of $b$. If $T$ is a standard Young tableau associated to $\lambda$ , let be $T(i)$ for the box $b$ of $\lambda$, in which $i$ appears. And define the function $\beta$ over the set of all boxes of $\lambda$ as follows:
\begin{center}
    $\beta(b)$ = $i$ if $b \in \lambda^{i}.$
\end{center}
We also define the \emph{charged content} $c(b)$ of a box $b$ of $\lambda$ by the equation
\begin{align*}
c(b)=ct(b)rc_0+d_{\beta(b)}.
\end{align*}

Now we are able to enunciate theorem 5.1 of \cite{GGAD}.
\begin{theorem} \label{teoag}
 If there is a non-zero morphism $\Delta(\lambda)\rightarrow \Delta(\mu)$ , then there are $T \in SYT(\lambda)$ and $U \in SYT(\mu)$ with
 $$\begin{array}{ccc}
     c(U(i))-c(T(i)) \in \mathbb{Z}_{\geq 0} & $and$ & c(U(i))-c(T(i))=\beta(U(i))-\beta(T(i))\mod r
   \end{array}.
 $$
 \end{theorem}

 We use this theorem to prove the necessary conditions for the existence of morphisms between standard modules and for the sufficient conditions we use the singular polynomials described in Propositions \ref{sing1}, Proposition \ref{sing2} and Proposition\ref{sing3}.

\begin{theorem} \label{th1}
The necessary and sufficient conditions for the existence of a morphism between standard modules for $G(r,1,2)$ are shown in the followings tables:

\begin{center}
\begin{tabular}{|c||c|c|c|c|c|c|}
\hline
    \hfill & $\Delta(\lambda_i)$ & $\Delta(\lambda_j)$ & $\Delta(\lambda^i)$ & $\Delta(\lambda^j)$ & $\Delta(\lambda_{i,j})$ & $\Delta(\lambda_{j,k})$ \\ \hline \hline
    $\Delta(\lambda_i)$ & $\cdot$ & $d_j-d_i$ & $c_0=-\frac{k}{2}$ & $\begin{array}{c}
                                                             d_j-d_i \\
                                                             c_0=-\frac{k}{2}
                                                           \end{array}$
     & $d_j-d_i-c_0r$ & $\begin{array}{c}
                        d_j-d_i \\
                        d_k-d_j-c_0r
                      \end{array}$
      \\ \hline
      $\Delta(\lambda^i)$ & $c_0=\frac{k}{2}$ & $\begin{array}{c}
                                    d_j-d_i \\
                                    c_0=\frac{k}{2}
                                  \end{array}$
     & $\cdot$ & $d_j-d_i$ & $d_j-d_i+c_0r$ & $\begin{array}{c}
                                                d_j-d_i \\
                                                d_k-d_j+c_0r
                                              \end{array}$\\
    \hline
  \end{tabular}
\end{center}
\begin{center}
\begin{adjustbox}{max width=\textwidth}
\begin{tabular}{|c||c|c|c|c|c|c|c|}
  \hline
  \hfill & $\Delta(\lambda_i)$ & $\Delta(\lambda^i)$ & $\Delta(\lambda_k)$ & $\Delta(\lambda^k)$ & $\Delta(\lambda_{i,j})$ & $\Delta(\lambda_{i,k})$ & $\Delta(\lambda_{k,s})$ \\ \hline \hline
  $\Delta(\lambda_{i,j})$ & $d_i-d_j+c_0r$ & $d_i-d_j-c_0r$ & $\begin{array}{c}
                                                       d_k-d_i \\
                                                       d_k-d_j+c_0r
                                                     \end{array}
  $ & $\begin{array}{c}
         d_k-d_i \\
         d_k-d_j-c_0r
       \end{array}
  $ & $\cdot$ & $d_k-d_j$ & $\begin{array}{c}
                                     d_k-d_i \\
                                     d_s-d_j \\
                                       $or$  \\
                                     d_s-d_i \\
                                     d_k-d_j
        \end{array}
  $ \\
  \hline
\end{tabular}
\end{adjustbox}
\end{center}

Columns represent the domain, rows represent the codomain and the entries represent conditions on the parameters. If more than one condition appears, this means that both must hold. The dots mean that there is no condition. The condition $d_i-d_j$ means that $d_i-d_j \in \mathbb{Z}_{\geq 0}$ and $d_i-d_j=i-j \mod r$. The condition $d_i-d_j\pm c_0r$ means $d_i-d_j\pm c_0r \in \mathbb{Z}_{\geq 0}$ and $d_i-d_j\pm c_0r=i-j \mod r$. Finally the conditions $c_0=\pm\frac{k}{2}$ say that $k$ is a positive odd integer.
\end{theorem}

\begin{proof}
For the necessary conditions we use Theorem \ref{teoag} attached to our $r$-partitions. This theorem give us almost all the conditions. Nevertheless in the cases of $\lambda_i\rightarrow \lambda^j$ and $\lambda^i\rightarrow \lambda_j$ the theorem shows in the first case that $c_0=-\frac{k}{2}$ and in the second case that $c_0=\frac{k}{2}$, without the condition that $k$ is odd. To get this condition we apply Theorem 1.2 of \cite{GGAD} with $G_{S}=G(1,1,2)$ to obtain a nonzero morphisms $\Delta_{c_0}(\mathrm{sign})\rightarrow \Delta_{c_0}(\mathrm{triv})$ for the rational Cherednik algebra for $G(1,1,2)$. This implies that $c_0=\frac{k}{2}$ with the condition that $k$ is odd.

To prove that this conditions are sufficient we construct explicit homomorphisms using the singular polynomials described in Propositions \ref{sing1}, Proposition \ref{sing2} and Proposition \ref{sing3}. We start with the cases that only have one condition.
\begin{enumerate}[1)]
  \item $\Delta(\lambda_i)\rightarrow \Delta(\lambda_j)$.\\
        In this case the condition is $d_j-d_i$. If we use $n=d_j-d_i$, we have the condition $(b)$ of Proposition \ref{sing1}. In this case the morphism is given by sending $1\otimes v_T\rightarrow x_1^nx_2^n\otimes v_T$.
  \item $\Delta(\lambda_i)\rightarrow \Delta(\lambda^i)$.\\
        In this case the condition is $c_0=-\frac{k}{2}$. We have the condition $(a)$ of Proposition \ref{sing2}. In this case the morphism is given by sending $1\otimes v_T\rightarrow (x_1^r-x_2^r)^k\otimes v_T$.
  \item $\Delta(\lambda_i)\rightarrow \Delta(\lambda_{i,j})$.\\
        In this case the condition is $d_j-d_i-c_0r$. We have two options:
         \begin{enumerate}
         \item $i<j$. If we use $n=d_j-d_i-c_0r$, we have the condition $(3.b)$ of Proposition \ref{sing3}. In this case the morphisms is given by sending $1\otimes v_T\rightarrow p(x_1,x_2)$ where $p(x_1,x_2)$ is the singular polynomials of case $(3.b)$ of Proposition \ref{sing3}.
         \item $i>j$. If we use $n=d_j-d_i-c_0r$, we have the condition $(3.a)$ of Proposition \ref{sing3}. In this case the morphisms is given by sending $1\otimes v_T\rightarrow p(x_1,x_2)$ where $p(x_1,x_2)$ is the singular polynomials of case $(3.a)$ of Proposition \ref{sing3}.
         \end{enumerate}
  \item $\Delta(\lambda^i)\rightarrow \Delta(\lambda_i)$.\\
        In this case the condition is $c_0=\frac{k}{2}$. We have the condition $(a)$ of Proposition \ref{sing1}. In this case the morphism is given by sending $1\otimes v_T\rightarrow (x_1^r-x_2^r)^k\otimes v_T$.
  \item $\Delta(\lambda^i)\rightarrow \Delta(\lambda^j)$.\\
        In this case the condition is $d_j-d_i$. If we use $n=d_j-d_i$, we have the condition $(b)$ of Proposition \ref{sing2}. In this case the morphism is given by sending $1\otimes v_T\rightarrow x_1^nx_2^n\otimes v_T$.
  \item $\Delta(\lambda^i)\rightarrow \Delta(\lambda_{i,j})$.\\
        In this case the condition is $d_j-d_i+c_0r$. We have two options:
         \begin{enumerate}
         \item $i<j$. If we use $n=d_j-d_i+c_0r$, we have the condition $(2.b)$ of Proposition \ref{sing3}. In this case the morphisms is given by sending $1\otimes v_T\rightarrow p(x_1,x_2)$ where $p(x_1,x_2)$ is the singular polynomials of case $(2.b)$ of Proposition \ref{sing3}.
         \item $i>j$. If we use $n=d_j-d_i-c_0r$, we have the condition $(2.a)$ of Proposition \ref{sing3}. In this case the morphisms is given by sending $1\otimes v_T\rightarrow p(x_1,x_2)$ where $p(x_1,x_2)$ is the singular polynomials of case $(2.a)$ of Proposition \ref{sing3}.
         \end{enumerate}
  \item $\Delta(\lambda_{i,j})\rightarrow \Delta(\lambda_i)$.\\
        In this case the condition is $d_i-d_j+c_0r$. If we use $n=d_i-d_j+c_0r$ we are in case $(c)$ of Proposition \ref{sing1}. We have two options:
        \begin{enumerate}
        \item $i<j$. The morphisms is given by sending $1\otimes v_{T_2}\rightarrow p(x_1,x_2)\otimes v_T$ where $p(x_1,x_2)$ is the singular polynomials of case $(c)$ of Proposition \ref{sing1}.
        \item $i>j$. The morphisms is given by sending $1\otimes v_{T_1}\rightarrow p(x_1,x_2)\otimes v_T$ where $p(x_1,x_2)$ is the singular polynomials of case $(c)$ of Proposition \ref{sing1}.
        \end{enumerate}
  \item $\Delta(\lambda_{i,j})\rightarrow \Delta(\lambda^i)$.\\
        In this case the condition is $d_i-d_j-c_0r$. If we use $n=d_i-d_j-c_0r$ we are in case $(c)$ of Proposition \ref{sing2}. We have two options:
        \begin{enumerate}
        \item $i<j$. The morphisms is given by sending $1\otimes v_{T_2}\rightarrow p(x_1,x_2)\otimes v_T$ where $p(x_1,x_2)$ is the singular polynomials of case $(c)$ of Proposition \ref{sing2}.
        \item $i>j$. The morphisms is given by sending $1\otimes v_{T_1}\rightarrow p(x_1,x_2)\otimes v_T$ where $p(x_1,x_2)$ is the singular polynomials of case $(c)$ of Proposition \ref{sing2}.
        \end{enumerate}
  \item $\Delta(\lambda_{i,j})\rightarrow \Delta(\lambda_{i,k})$.\\
        In this case the condition is $d_k-d_j$. We have four options:
        \begin{enumerate}
          \item $i<j$ and $i<k$. If we use $n=d_k-d_j$, we are in case $(1.b)$ of Proposition \ref{sing3}. In this case the morphisms is given by sending $1\otimes v_{T_1}\rightarrow p(x_1,x_2)$ where $p(x_1,x_2)$ is the singular polynomials of case $(1.b)$ of Proposition \ref{sing3}.
          \item $i<j$ and $i>k$. If we use $n=d_k-d_j$, we are in case $(1.a)$ of Proposition \ref{sing3}. In this case the morphisms is given by sending $1\otimes v_{T_2}\rightarrow p(x_1,x_2)$ where $p(x_1,x_2)$ is the singular polynomials of case $(1.a)$ of Proposition \ref{sing3}.
          \item $i>j$ and $i<k$. If we use $n=d_k-d_j$, we are in case $(1.b)$ of Proposition \ref{sing3}. In this case the morphisms is given by sending $1\otimes v_{T_2}\rightarrow p(x_1,x_2)$ where $p(x_1,x_2)$ is the singular polynomials of case $(1.b)$ of Proposition \ref{sing3}.
          \item $i>j$ and $i>k$. If we use $n=d_k-d_j$, we are in case $(1.a)$ of Proposition \ref{sing3}. In this case the morphisms is given by sending $1\otimes v_{T_1}\rightarrow p(x_1,x_2)$ where $p(x_1,x_2)$ is the singular polynomials of case $(1.a)$ of Proposition \ref{sing3}.
        \end{enumerate}
\end{enumerate}
We continue with the cases that have two conditions. There are seven cases with two conditions:
\begin{enumerate}
  \item $\Delta(\lambda_i)\rightarrow \Delta(\lambda^j)$ or $\Delta(\lambda^i)\rightarrow \Delta(\lambda_j)$.\\
   For $\Delta(\lambda_i)\rightarrow \Delta(\lambda^j)$ we have the conditions $d_j-d_i$ and $c_0=-\frac{k}{2}$. The condition $c_0=-\frac{k}{2}$ allows the construction of the morphism $\Delta(\lambda_i)\rightarrow \Delta(\lambda^i)$. The condition $d_j-d_i$ allows the construct of the morphism $\Delta(\lambda^i)\rightarrow \Delta(\lambda^j)$. The composition of these two morphisms is a morphism from $\Delta(\lambda_i)$ to  $\Delta(\lambda^j)$. This is a non-zero composition, because it is of the form $1\otimes v_T \rightsquigarrow pq\otimes v_T$ where $p$ and $q$ are non-zero polynomials. For $\Delta(\lambda^i)\rightarrow \Delta(\lambda_j)$ we use the same arguments as before attached to this case.

  \item $\Delta(\lambda_i)\rightarrow \Delta(\lambda_{j,k})$ or $\Delta(\lambda^i)\rightarrow \Delta(\lambda_{j,k})$.\\
   For $\Delta(\lambda_i)\rightarrow \Delta(\lambda_{j,k})$ we have the conditions $d_j-d_i$ and $d_k-d_j-c_0r$. The condition $d_j-d_i$ allows the construction of the morphism $\Delta(\lambda_i)\rightarrow \Delta(\lambda_j)$. The condition $d_k-d_j-c_0r$ allow the construction of the morphism $\Delta(\lambda_j)\rightarrow \Delta(\lambda_{j,k})$. The composition of these two morphisms is a morphism from $\Delta(\lambda_i)$ to $\Delta(\lambda_{j,k})$. This is a non-zero composition, because it is of the form $1\otimes v_T \rightsquigarrow pq\otimes v_{T_1}+pr\otimes v_{T_2}$ where $p,q,r$ are non-zero polynomials. For $\Delta(\lambda^i)\rightarrow \Delta(\lambda_{j,k})$ we use the same arguments as before attached to this case.

  \item $\Delta(\lambda_{i,j})\rightarrow \Delta(\lambda_k)$ or $\Delta(\lambda_{i,j})\rightarrow \Delta(\lambda^k)$.\\
  For $\Delta(\lambda_{i,j})\rightarrow \Delta(\lambda_k)$ we have the conditions $d_k-d_i$ and $d_k-d_i+c_0r$. The condition $d_k-d_i$ allows the construction of the morphism $\Delta(\lambda_{i,j})\rightarrow \Delta(\lambda_{j,k})$. The condition $d_k-d_i+c_0r$ allow the construction of the morphism $\Delta(\lambda_{j,k})\rightarrow \Delta(\lambda_k)$. The composition of these two morphisms is a morphism from $\Delta(\lambda_{i,j})$ to $\Delta(\lambda_k)$. This composition is of the form $1\otimes v_{T_1} \rightsquigarrow (pr+qr^\prime) \otimes v_{T}$ where $r^\prime$ is just interchanging $x_1$ and $x_2$ in $r$. Looking at the coefficients of the polynomials involved we can see that $(pr+qr^\prime)$ is a non-zero polynomial . For $\Delta(\lambda_{i,j})\rightarrow \Delta(\lambda^k)$ we use the same arguments as before attached to this case.

  \item $\Delta(\lambda_{i,j})\rightarrow \Delta(\lambda_{k,s})$.\\
   For this case we have the conditions $d_k-d_i$ and $d_s-d_j$ (or $d_s-d_i$ and $d_k-d_j$). The condition $d_k-d_i$ allows the construction of the morphism  $\Delta(\lambda_{i,j})\rightarrow \Delta(\lambda_{k,j})$. The condition $d_s-d_j$ allows the construction of the morphisms $\Delta(\lambda_{k,j})\rightarrow \Delta(\lambda_{k,s})$. The composition of these two morphisms is a morphism from $\Delta(\lambda_{i,j})$ to $\Delta(\lambda_{k,s})$. This composition is of the form $1\otimes v_{T_1} \rightsquigarrow (pr+qr^\prime) \otimes v_{T_1}+(ps+qs^\prime)\otimes v_{T_2}$ where $r^\prime$ and $s^\prime$ is just interchanging $x_1$ and $x_2$ in $r$ and $s$. Looking at the coefficients of the polynomials involved we can see that $(pr+qr^\prime)$ or $(ps+qs^\prime)$ is a non-zero polynomial. For the condition $d_s-d_i$ and $d_k-d_j$ we can do the same as before.
  \end{enumerate}
\end{proof}

\section{Dimension} \label{SecDim}

In this section we establish sufficient conditions to have that the dimension of the homomorphisms space between two standard modules is two. We suspect that these sufficient conditions are also necessary conditions for having a two dimensional space of morphisms of any standard module.
\begin{theorem} If we have the following conditions  \begin{itemize}
                                                            \item $d_i-d_k+c_0r\in\ZZ_{\geq 0}$ and $d_i-d_k+c_0r\equiv i-k$ mod $r$.
                                                            \item $d_i-d_k-c_0r\in\ZZ_{\geq 0}$ and $d_i-d_k-c_0r\equiv i-k$ mod $r$.
                                                            \item $d_j-d_i+c_0r\in\ZZ_{\geq 0}$ and $d_j-d_i+c_0r\equiv i-k$ mod $r$.
                                                            \item $d_j-d_i-c_0r\in\ZZ_{\geq 0}$ and $d_j-d_i-c_0r\equiv j-1$ mod $r$.
                                                          \end{itemize}
where $c_0$ is a non-zero integer, then we have that $$\mathrm{Dim}(\mathrm{Hom}_{\mathbb{H}}(\Delta(\lambda_{i,k}),\Delta(\lambda_{i,j})))=2.$$
\end{theorem}
\begin{proof}
We have that this fourth condition allows the construction of morphisms between some standard modules. In particular we have that
\begin{center}
\begin{tabular}{|c|c|c|}
  \hline
  1 & $d_i-d_k+c_0r$ & $\Delta(\lambda_{i,k})\rightarrow \Delta(\lambda_i)$ \\
  2 & $d_i-d_k-c_0r$ & $\Delta(\lambda_{i,k})\rightarrow \Delta(\lambda^i)$ \\
  3 & $d_j-d_i+c_0r$ & $\Delta(\lambda^{i})\rightarrow \Delta(\lambda_{i,j})$ \\
  4 & $d_j-d_i-c_0r$ & $\Delta(\lambda_{i})\rightarrow \Delta(\lambda_{i,j})$ \\
  \hline
\end{tabular}
\end{center}
 We can see that we have two ways to go from $\Delta(\lambda_{i,k})$ to $\Delta(\lambda_{i,j})$. We are proving that these two ways are linearly independent. For this we see the leading terms of each of these morphisms. In order to compute the leading terms of the singular polynomials involved, we need to consider that if $c_0$ is an integer, it could change the leading terms. Suppose that $c_0>0$. The leading term can be calculated using the singular polynomials:
\begin{itemize}
  \item $x_1^{d_j-d_i+c_0r-l^\prime r}x_2^{l^\prime r}$ for case 1.
  \item $x_1^{d_j-d_i-c_0r}$ for case 2.
  \item $x_1^{lr}x_2^{d_i-d_k+c_0r-lr}\otimes v_{T_1}-x_1^{d_i-d_k+c_0r-lr}x_2^{lr}\otimes v_{T_2}$ for case 3.
  \item $x_2^{d_i-d_k-c_0r}\otimes v_{T_1}+x_1^{d_i-d_k-c_0r}\otimes v_{T_2}$ for case 4.
\end{itemize}
where $l$ and $l^\prime$ are integers. The composition of the morphisms follows by multiplying the polynomials. The leading terms of the compositions are:
\begin{itemize}
\item For the composition $4\circ 1$
\begin{center}
$(x_1^{d_j-d_i+c_0r-l^\prime r}x_2^{l^\prime r})(x_2^{d_i-d_k-c_0r}\otimes v_{T_1}+x_1^{d_i-d_k-c_0r}\otimes v_{T_2})=$ $x_1^{d_j-d_i+c_0r-l^\prime r}x_2^{d_i-d_k-c_0r+l^\prime r}\otimes v_{T_1}+x_1^{d_j-d_k-l^\prime r}x_2^{l^\prime r}\otimes v_{T_2}$
 \end{center}
 \item For the composition $3\circ 2$
 \begin{center}
 $(x_1^{d_j-d_i-c_0r})(x_1^{lr}x_2^{d_i-d_k+c_0r-lr}\otimes v_{T_1}-x_1^{d_i-d_k+c_0r-lr}x_2^{lr}\otimes v_{T_2})=$ $x_1^{d_j-d_i-c_0r+lr}x_2^{d_i-d_k+c_0r-lr}\otimes v_{T_1}-x_1^{d_j-d_k-lr}x_2^{lr}\otimes v_{T_2}.$
 \end{center}
\end{itemize}
 If we compare this two terms we can see that they are linearly independent. In conclusion  we have two linearly independent ways to go from $\Delta(\lambda_{i,k})$ to $\Delta(\lambda_{i,j})$. This implies that the dimension of the space of homomorphism is 2.
\end{proof}
\section{Example} \label{SecEx}
In this section we give an explicit example.
\begin{example}
For this example we work with $r=3$. Suppose first that $10-d_0+d_2=0$. This is a condition of the form $d_0-d_2$ and allows the construction of some morphisms.
   Next we add the condition $5-d_0+d_1=0$. This is a condition from the form $d_0-d_1$. With these two conditions we can form a new one by subtracting the second condition from the first one. This new condition is $5-d_1+d_2=0$ and is from the form $d_1-d_2$. This way it is possible to construct more morphisms.
  Finally, if we add the condition $c_0=1$, then we have 6 new conditions
\begin{center}
    $\begin{array}{rcr}
       13-d_0+d_2-c_0r=0 & \hfill & (d_0-d_2+c_0r) \\
       7-d_0+d_2+c_0r=0 & \hfill & (d_0-d_2-c_0r) \\
       8-d_0+d_1-c_0r=0 & \hfill & (d_0-d_1+c_0r) \\
       2-d_0+d_1+c_0r=0 & \hfill & (d_0-d_1-c_0r) \\
       8-d_1+d_2-c_0r=0 & \hfill & (d_1-d_2+c_0r) \\
       2-d_1+d_2+c_0r=0 & \hfill & (d_1-d_2-c_0r)
     \end{array}
    $\end{center}
and this allows the construction of 12 new morphisms. The following table shows all the morphisms constructed by the corresponding conditions.\\
\begin{center}
\begin{tabular}{|c|c|}
      \hline
      $d_0-d_2$ & $\begin{array}{lll|l}
                 \Delta(\lambda_2) & \rightarrow & \Delta(\lambda_0) & 1\\
                 \Delta(\lambda^2) & \rightarrow & \Delta(\lambda^0) & 2\\
                 \Delta(\lambda_{1,2}) & \rightarrow & \Delta(\lambda_{0,1}) & 3
               \end{array}$ \\ \hline
      $\begin{array}{c}
         \textcolor[rgb]{0.98,0.00,0.00}{d_0-d_1} \\
         \textcolor[rgb]{0.98,0.00,0.00}{d_1-d_2}
       \end{array}
      $ & $\textcolor[rgb]{0.98,0.00,0.00}{\begin{array}{lll|l}
             \Delta(\lambda_2) & \rightarrow & \Delta(\lambda_1) & 4\\
             \Delta(\lambda_1) & \rightarrow & \Delta(\lambda_0) & 5\\
             \Delta(\lambda^2) & \rightarrow & \Delta(\lambda^1) & 6\\
             \Delta(\lambda^1) & \rightarrow & \Delta(\lambda^0) & 7 \\
             \Delta(\lambda_{1,2}) & \rightarrow & \Delta(\lambda_{0,2}) & 8 \\
             \Delta(\lambda_{0,2}) & \rightarrow & \Delta(\lambda_{0,1}) & 9
           \end{array}}
      $ \\ \hline
      $\textcolor[rgb]{0.00,0.00,1.00}{\begin{array}{c}
                                              c_0=1 \\
                                         d_0-d_2+c_0r \\
                                         d_0-d_2-c_0r \\
                                         d_0-d_1+c_0r \\
                                         d_0-d_1-c_0r \\
                                         d_1-d_2+c_0r \\
                                         d_1-d_2-c_0r
                                       \end{array}
      }$ & $\textcolor[rgb]{0.00,0.00,1.00}{\begin{array}{lll|r}
                                              \Delta(\lambda_{0,1}) & \rightarrow & \Delta(\lambda_0) & 10 \\
                                              \Delta(\lambda_2) & \rightarrow & \Delta(\lambda_{1,2}) & 11 \\
                                              \Delta(\lambda_{0,1}) & \rightarrow & \Delta(\lambda^0) & 12 \\
                                              \Delta(\lambda^2) & \rightarrow & \Delta(\lambda_{1,2}) & 13 \\
                                              \Delta(\lambda_2) & \rightarrow & \Delta(\lambda_{0,2}) & 14 \\
                                              \Delta(\lambda_{1,2}) & \rightarrow & \Delta(\lambda_1) & 15 \\
                                              \Delta(\lambda_{1,2}) & \rightarrow & \Delta(\lambda^1) & 16 \\
                                              \Delta(\lambda^1) & \rightarrow & \Delta(\lambda_{0,1}) & 17 \\
                                              \Delta(\lambda_1) & \rightarrow & \Delta(\lambda_{0,1}) & 18 \\
                                              \Delta(\lambda_{0,2}) & \rightarrow & \Delta(\lambda_0) & 19 \\
                                              \Delta(\lambda_{0,2}) & \rightarrow & \Delta(\lambda^0) & 20 \\
                                              \Delta(\lambda^2) & \rightarrow & \Delta(\lambda_{0,2}) & 21
                                                                                          \end{array}
      }$ \\
      \hline
    \end{tabular}
   \end{center}

In this last table we have enumerated the morphisms and we obtain the following diagram\\
\begin{center}
    \includegraphics[scale=0.8]{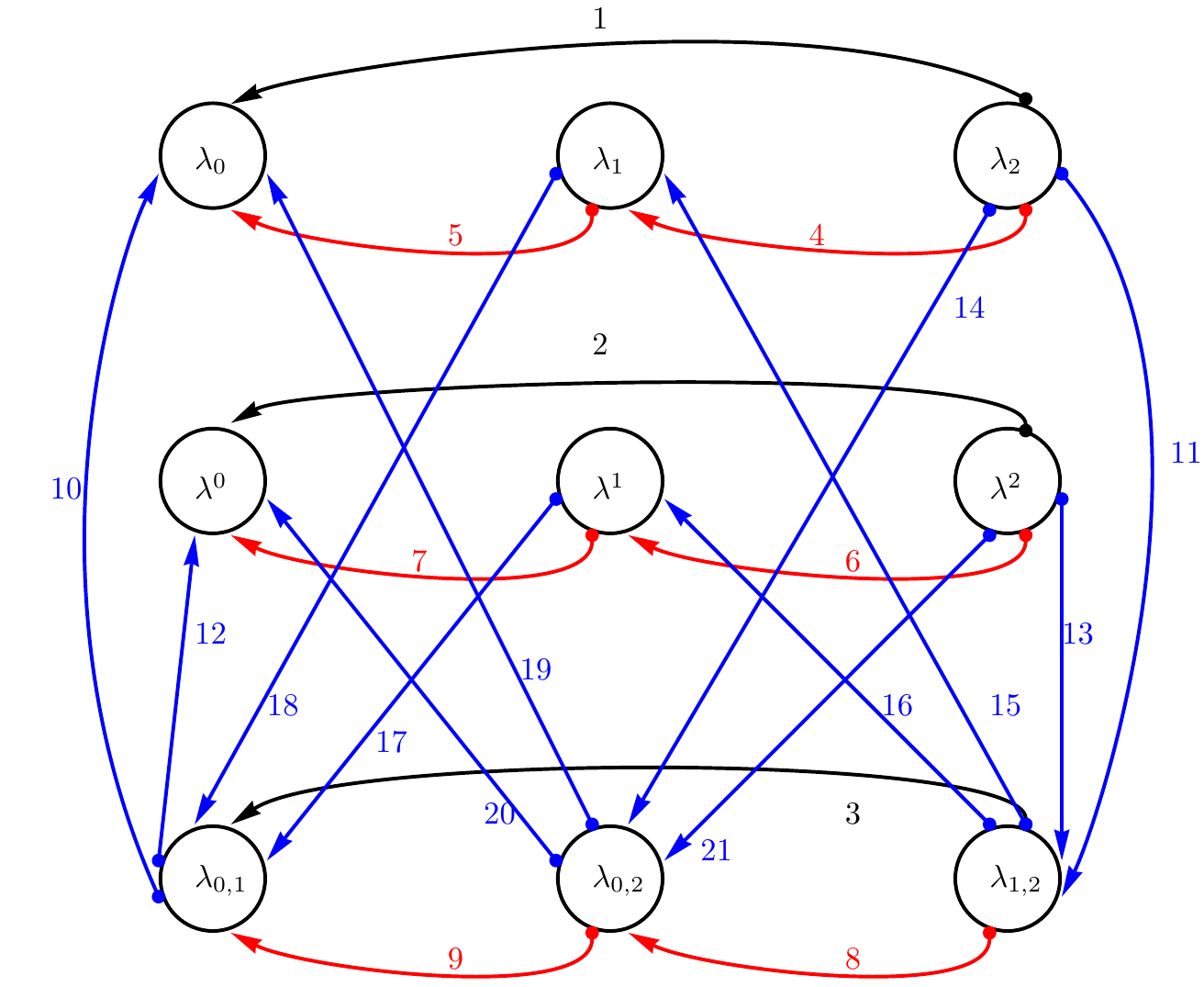}\\
\end{center}
Now we describe each of the 21 morphisms using the singular polynomials. All the computations are using the three imposed conditions. If we delete one of the conditions, the polynomials could change.
\begin{center}
\begin{tabular}{|l|l|}
  \hline
  1 & $x_1^{10}x_2^{10}$ \\
  2 & $x_1^{10}x_2^{10}$  \\
  3 & $x_1^{5}x_2^{5}\otimes v_{T_1}$  \\
  4 & $x_1^{5}x_2^{5}$  \\
  5 & $x_1^{5}x_2^{5}$  \\
  6 & $x_1^{5}x_2^{5}$  \\
  7 & $x_1^{5}x_2^{5}$  \\
  8 & $(x_1^{5}+\frac{1}{6}x_1^2x_2^3)\otimes v_{T_1}+(\frac{1}{3}x_1^4x_2+\frac{1}{2}x_1x_2^4)\otimes v_{T_2}$  \\
  9 & $(x_2^{5}+\frac{1}{6}x_1^3x_2^2)\otimes v_{T_1}-(\frac{1}{2}x_1x_2^4+\frac{1}{3}x_1^4x_2)\otimes v_{T_2}$  \\
  10 & $x_1^8-x_1^2x_2^6-2x_1^5x_2^3$  \\
  11 & $x_1^2\otimes v_{T_1}+x_2^2\otimes v_{T_2}$  \\
  12 & $x_1^2$  \\
  13 & $(x_1^8-3x_1^5x_2^3-x_1^2x_2^6)\otimes v_{T_1}-(x_2^8-3x_1^3x_2^5-x_1^6x_2^2)\otimes v_{T_2}$  \\
  14 & $(x_1^7+\frac{2}{3}x_1^4x_2^3+\frac{1}{3}x_1x_2^6)\otimes v_{T_1}-(x_2^7+\frac{2}{3}x_1^3x_2^4+\frac{1}{3}x_1^6x_2)\otimes v_{T_2}$  \\
  15 & $x_1^8-x_1^2x_2^6-2x_1^5x_2^3$  \\
  16 & $x_1^2$  \\
  17 & $(x_1^8-2x_1^5x_2^3-x_1^2x_2^6)\otimes v_{T_1}-(x_2^8-2x_1^3x_2^5-x_1^6x_2^2)\otimes v_{T_2}$  \\
  18 & $x_1^2\otimes v_{T_1}+x_2^2\otimes v_{T_2}$  \\
  19 & $x_1^{13}-\frac{1}{3}x_1x_2^{12}-\frac{4}{3}x_1^{10}x_2^3$  \\
  20 & $x_1^7+\frac{1}{5}x_1x_2^6-\frac{2}{5}x_1^4x_2^3$  \\
  21 & $(x_1^{13}-\frac{1}{3}x_1x_2^{12}-\frac{4}{3}x_1^{10}x_2^3)\otimes v_{T_1}-(x_2^{13}-\frac{1}{3}x_1^{12}x_2-\frac{4}{3}x_1^{3}x_2^{10})\otimes v_{T_2}$  \\
  \hline
\end{tabular}
\end{center}

There are many morphisms that can be constructed using other morphisms. If we delete from the diagram all the morphisms that come from other morphisms, we will get the following diagram
\begin{center}
\includegraphics[scale=0.8]{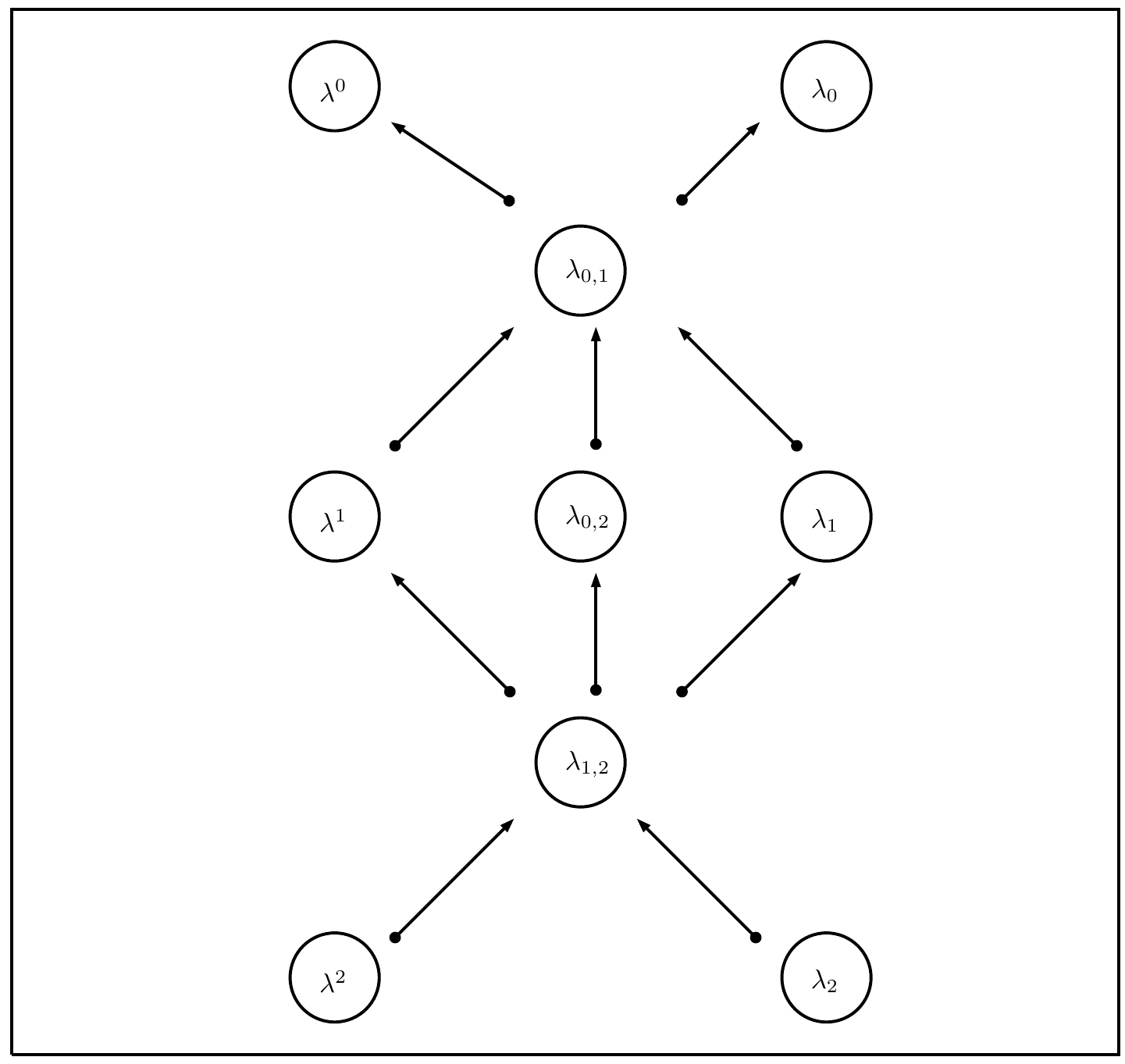}
\end{center}

For the three morphisms from $\Delta(\lambda_{1,2})$ to $\Delta(\lambda_{0,1})$, only two of them are linearly independent. We can also see that the diagram is self-dual and graded. Is it true that the canonical coarsest order is always self-dual and graded?
\end{example}

\newpage
\bibliographystyle{amsplain}
\def\cprime{$'$} \def\cprime{$'$}

\end{document}